\documentclass[preprint,12pt]{article}
\usepackage{cite} 
\usepackage{color}
\usepackage{graphicx}
\usepackage{amssymb}  
\usepackage{amsmath}  
\usepackage{amsthm}
\usepackage{hyperref}

\newcommand{\R}{{\mathbb R}}

\newcommand{\cyr}{%
  \renewcommand\rmdefault{wncyr}%
  \renewcommand\sfdefault{wncyss}%
  \renewcommand\encodingdefault{OT2}%
  \normalfont\selectfont}
\DeclareTextFontCommand{\textcyr}{\cyr}

\newcommand{\Cfvar}[1]{\text{C}^{#1}}
\newcommand{\Cfone}]{\ifmmode\Cfvar{1}\fi}
\newcommand{\Cftwo}]{\ifmmode\Cfvar{2}\fi}

\newcommand{\Dn}{\mathcal{D}(\R^n)}
\newcommand{\Dnminus}{\mathcal{D}(\R^{n-1})}

\newcommand{\Done}{\mathcal{D}(\R)}
\newcommand{\Embed}[2]{J_{#1}(#2)}
\newcommand{\EmbedMap}[1]{J_{#1}}

\newcommand{\Prj}{\Pi}

\newcommand{\QuSubD}[2]{D_Q #1(#2)}

\newcommand{\Sph}{\mathcal{S}}
\newcommand{\Sphn}{{\mathcal{S}_{n-1}}} 
\newcommand{\Sn}{\Sphn}
\newcommand{\Snvar}[1]{\Sph_{#1-1}}

\newcommand{\co}{\operatorname{co}}

\newcommand{\direct}[1]{\overrightarrow{#1}}

\newcommand{\hfre}{\hfill\mbox{}}

\newcommand{\oldminus}{\ominus}

\newcommand{\skal}[2]{\langle #1, #2 \rangle}
\newcommand{\supf}[2]{\delta^*(#1,#2)}

\newcommand{\trans}{\top}

\def\co{\mathop{\textrm{co}}}
\def\cl{\mathop{\textrm{cl}}}

\renewcommand\int{\mathop{{\rm int}}}

\renewcommand{\a}{\alpha}
\renewcommand{\b}{\beta}

\renewcommand{\phi}{\varphi}

\newcommand{\myspan}{\textrm{span\,}}

\newcommand{\dd}{{\overrightarrow \partial}}

\newif\ifColPlotsUsed
\ColPlotsUsedtrue

\newtheorem{theorem}{Theorem}[section]
\newtheorem{corollary}[theorem]{Corollary}
\newtheorem{definition}[theorem]{Definition}

\newtheorem{lemma}[theorem]{Lemma}
\newtheorem{prop}[theorem]{Proposition}
\newtheorem{remark}[theorem]{Remark}

\newcounter{Examplecount}
\definecolor{cherry}{rgb}{0.7,0,0.2}

\title{From Quasidifferentiable to Directed Subdifferentiable Functions: 
         \\
       Exact Calculus Rules}
	
\author{
Robert Baier%
\thanks{Chair of Applied Mathematics, University of Bayreuth,
                    95440 Bayreuth, Germany, robert.baier@uni-bayreuth.de}\,,
Elza Farkhi%
\thanks{School of Mathematical Sciences, Sackler Faculty of
                    Exact Sciences, Tel Aviv University, 69978 Tel Aviv, Israel, elza@post.tau.ac.il (on leave from the Institute of Mathematics and Informatics of the Bulgarian Academy of Sciences)}\,
and Vera Roshchina%
\thanks{School of Science, RMIT University, Melbourne Vic 3001, Australia, vera.roshchina@rmit.edu.au}
}

\begin{document}

\maketitle

\vspace{-0.4cm}
\centerline{Dedicated to the memory of V.~F.~Demyanov}

\begin{abstract}
   We derive exact calculus rules for the directed subdifferential 
   defined for the class of directed subdifferentiable functions. 
   We also state optimality conditions, a chain rule
   and a mean-value theorem. Thus we extend the theory of the 
   directed subdifferential from quasidifferentiable to directed 
   subdifferentiable functions.
\end{abstract}

   

\section{Introduction}

This paper is a continuation of \cite{BaiFarRosh2014a} in which we showed that the directed subdifferential 
introduced for \emph{differences of convex (DC, delta-convex)} functions from $\R^n$ to $\R$  \cite{BaierFarkhiDC} 
and extended for \emph{quasidifferentiable (QD)} functions \cite{BaiFarRosh2012a,BaiFarRosh2012b}, can be constructed 
from the directional derivatives only, without any information on the delta-convex structure of the function.
Both DC and QD functions are a continued matter of research in subdifferential calculus, see 
e.g.~\cite{HirUrr1985,DEM/JEY:1997,GAO:2004,AmaPenSya2008,DinMorNgh2009}.

Our work is strongly motivated and connected to the research of Vladimir Demyanov and his collaborators (see e.g.~\cite{DemRub1995}), 
who pioneered the study of constructive tools of nonsmooth analysis. The foundations of our developments can be found in 
quasidifferential calculus, but the main contribution of this work is that we eliminate the need for the DC representation 
of the directional derivative that is crucial for the construction of quasidifferential.
Thus we extend further the family of functions for which the directed subdifferential can be defined and
label the functions of the extended class as \emph{directed subdifferentiable}.  This class contains 
locally Lipschitz functions definable on o-minimal structures and quasidifferentiable functions (as well as 
such important subclasses as amenable functions and their negatives and lower/upper-$C^k$ functions for $k\ge 1$), see~\cite{BaiFarRosh2014a}. 
The directed subdifferential is visualized in $\R^n$ as a generally nonconvex set called 
\emph{Rubinov subdifferential}. The exact place of the Rubinov subdifferential in the chain of inclusions 
between other known subdifferentials is investigated in \cite{BaierFarkhiDC,BaiFarRosh2012b}. 
We only remind here that its convex hull is the Michel-Penot subdifferential (which is a subset 
of the Clarke subdifferential), and the Dini subdifferential is a convex part of it. 
All these known subdifferentials do not enjoy exact calculus rules in the nonconvex case
in contrast to the directed one. 
As we show here, the directed subdifferential enjoys exact calculus rules in the general class of directed subdifferentiable functions, 
and the same sharp optimality conditions as the Demyanov-Rubinov quasidifferential or the Dini subdifferential and superdifferential. 
The optimality conditions presented here are formulated using the order in the space of directed sets, 
but may be easily expressed by inclusion of the zero in the Rubinov subdifferential, as it is done in 
\cite{BaierFarkhiDC,BaiFarRosh2012a,BaiFarRosh2012b}.


Note that the directional derivative \cite{HirUrr1985,DemRub1995}  is a 
main ingredient of the directed subdifferential 
as well as of the constructed calculus rules and
optimality conditions.
The exact calculus rules for the directional derivatives allow to
extend automatic differentiation algorithms to 
structured nonsmooth 
functions, such as min- and max-type functions and their compositions
(see~\cite{Griewank2011,Griewank2013,BeckMoseNaum2012,KhanBarton2012}). 
In some of these works, the propagation of nonsmooth convex or concave
McCormick relaxations \cite{McCorm1976} are used 
and subgradient propagation as e.g.~in~\cite{MitsChachBar2009} is applied.
In contrary to this approach with McCormick relaxations, where only a subset of the entire subdifferential 
can be obtained for some examples as stated in~\cite[Sec.~4.3]{MitsChachBar2009},
we rely on exact calculus rules for the directed subdifferential, which is more in-line with
the approach of V.~Demyanov and A.~Rubinov.

The inductive construction of the space of directed subdifferentiable functions with respect to the dimension 
is based on the same principle as the construction of the space of \emph{directed sets}
and is reflected by the inductive proofs of our main results.
As the space of directed sets contains differences of 
embedded convex compact subsets of $\R^n$,  
 directed subdifferentiable functions include DC and QD functions.

The paper is organized as follows. The calculus rules 
for algebraic operations on directed subdifferentiable functions
as well as for pointwise maximum and minimum are derived
in Section~\ref{SecMoreCalcRules}; in Section \ref{secOptCond} we study necessary 
optimality conditions for directed subdifferentiable function, and Section~\ref{SecChainMeanVal} contains results related to a chain rule and a mean-value theorem for this function class.

We use the standard notation: by $\|\cdot \|$ we denote the Euclidean norm, 
and $\Sn$ is the unit sphere in $\R^n$.
We denote by $\cl{A}$ the \emph{closure} of the set $A$ and by $\co{A}$ its \emph{convex hull}.

\section{Calculus Rules} 
\label{SecMoreCalcRules}

The functions we are working with are a bit more general than tame (definable) locally Lipschitz ones. Readers familiar with tame or semialgebraic geometry may keep this in mind, as other examples of directed subdifferentiable functions are exotic and are unlikely to be met in applications; on the other hand, delta-convex and quasidifferentiable functions are also directed subdifferentiable \cite{BaiFarRosh2014a}.

A function from $\R^n$ to $\R$ is directed subdifferentiable if its Dini directional derivative restricted to lower-dimensional subspaces is $(n-1)$ times (recursively) directionally differentiable on $\R^n$, and all such recursive directional derivatives are uniformly bounded. We will soon give a precise definition.

For a function $f:\R^n\to \R$ we denote by $\nabla f(x)=f'(x)^T$ the \emph{gradient} of $f$ at 
the point $x \in \R^n$, and by $f'(x;l)$ we denote the 
\emph{(Dini) directional derivative} of $f$ at $x$ in the
direction $l \in \R^n$,
$$
f'(x;l):= \lim_{t\downarrow 0}\frac{f(x+tl) - f(x)}{t}.
$$
The function is called \emph{directionally
differentiable} at $x$ (see~\cite[Chap.~I, \S 3.1]{DemRub1995} and \cite{Schol1994}), 
if 
the limit $f'(x;l)$ at $x$ exists for all $l \in \Sn$.

Denote by $\Prj_{n-1,l}$ a fixed linear transformation from $\R^n$ to $\R^{n-1}$ that isometrically maps the orthogonal complement of the vector $l$ to $\R^{n-1}$, and maps $l$ to zero, see \cite{BaiFar2001,BaiFarRosh2012a}. The mapping $\Prj^T_{n-1,l}$ is the inverse of $\Prj_{n-1,l}$, and maps $\R^{n-1}$ onto the $(n-1)$-dimensional subspase of $\R^n$ that is orthogonal to $l$.

\begin{definition}[\bf directed subdifferentiable function] \label{def:dsf}
A function $f:\R \to \R$ is called \emph{$M$-directed subdifferentiable} at $x\in \R$ if both its left- and right-sided derivative exist and
$$
   \max \{|f'(x;-1)|,|f'(x;1)|\}\leq M \;.
$$
A function $f:\R^n\to \R$, $n\ge 2$, is called \emph{$M$-directed subdifferentiable} at $x\in \R^n$ if its directional derivative $f'(x;l)$ exists for every $l$, is continuous as a function of the direction $l$, is bounded by $M$, i.e.
$$
\max_{l\in \Sn} |f'(x;l)|\leq M ,
$$
and the restriction
$f_l(\cdot) = f'(x;l+\Prj^\trans_{n-1,l}(\cdot)):\R^{n-1}\to \R$ of its directional derivative to $l+\myspan \{l\}^\perp$ is also $M$-directed subdifferentiable at $0_{n-1}$ for all $l\in \Sn$.

We say that a function is {\em directed subdifferentiable} if it is $M$-directed subdifferentiable for some $M\geq 0$.
\end{definition}

Observe that every univariate directionally differentiable function $f:\R\to \R$ is directed subdifferentiable, as we only need to consider the `first order' directional derivatives which are trivially bounded.

The directed subdifferential is an object in the Banach space of {\em directed sets}, defined in~\cite{BaiFar2001,BaiFar2001b}. Directed sets are defined recursively in the dimension: a one-dimensional directed set is an ordered pair of two numbers, and for $n \geq 2$ a directed set maps $\Sn$ to a two-component object that consists of a directed set of dimension $n-1$ and a real number (the value of the generalized support function). The real-valued component must be continuous, and the directed set component must be uniformly bounded on $\Sn$ (in a recursively defined norm induced by the lower-dimensional directed sets and the lower-dimensional generalized support functions). The basic facts about the Banach space of directed sets can be found in~\cite{BaiFarRosh2012a,BaiFarRosh2014a}, and more details are available in~\cite{BaiFar2001,BaiFar2001b}. We provide the definition of the directed set and some other necessary definitions for convenience, even though this may make the paper appear a bit
  repetitive.

\begin{definition}
   We call $\direct{A}$  a {\em directed set}
    \begin{itemize}
       \item[(i)]  in $\R$, if it is a \emph{directed interval} $\direct{A}= \direct{[-a_1(-1),a_1(1)]}:=(a_1(-1), a_1(1))$. Its \emph{norm} is
                    $\| \direct{A} \|_1 =
                   \max\limits_{ l = \pm 1 } | a_1(l) |$,
       \item[(ii)] in $\R^n,$ $n \geq 2$, if there exist
                   a continuous function $a_n: \Sn \rightarrow \R$
                   (named as \emph{generalized support function})
                   and a map having
                   lower-dimensional directed sets as images
                   $\direct{A_{n-1}}: \Sn \rightarrow
                   \Dnminus$ which is uniformly bounded
                   with respect to $\| \cdot \|_{n-1}$.
    \end{itemize}
   We denote $\direct{A} = (\direct{A_{n-1}(l)}, a_n(l))_{l \in \Sn}$ and define its \emph{norm} as
    $$
       \| \direct{A} \| :=
       \| \direct{A} \|_n := \max\{ \sup\limits_{ l \in \Sn }
                                    \| \direct{A_{n-1}(l)} \|_{n-1},
                                    \max\limits_{ l \in \Sn } | a_n(l) | \} \;.
    $$
   The set of all directed sets in $\R^n$ is denoted by $\Dn$.

The \emph{embedding} $\EmbedMap{n}$ of convex compact sets into the space of directed sets $\Dn$ is defined as follows. We distinguish two cases, i.e.
\begin{itemize}
       \item[(i)]  in $\R$, if $A=[a^l,a^r]$, then the embedded directed interval is 
			$\direct{A}= \Embed{1}{ A } := (-a^l,a^r)$ 
                        which is also denoted as $\direct{[a^l,a^r]}$,
       \item[(ii)] in $\R^n,$ $n \geq 2$, for a convex compact $A$,  
			$\direct{A} = \Embed{n}{ A }$ is defined as the family of pairs
                        $(\direct{A_{n-1}(l)}, \delta^*(l,A))_{l \in \Sn}$,
    	                where 
                         \[
                            A_{n-1}(l):=\Prj_{n-1,l}(\rm{arg}\max_{x\in A} \,\skal{l}{x}) \quad (l \in \Sn), 
                         \]
                        and $\rm{arg}\max_{x\in A} \,\skal{l}{x}$ is
			the $(n-1)$-dimensional \emph{supporting} \emph{face} of $A$ and 
                         \[
			    \delta^*(l,A) =\max_{x\in A} \,\skal{l}{x}
                         \]
                        is the \emph{support function} of $A$.
\end{itemize}
\end{definition}

Thus, the convex compacts in $\R^n$ are embedded in the space of directed sets $\Dn$ with the help
of their support functions and supporting faces. Conversely, a directed set in $\Dn$ can be mapped
               to a compact set in $\R^n$ with the help of the visualization mapping $V_n$. Details on the visualization may
be found in~\cite{BaiFar2001b,BaiFarRosh2012a,BaiFarRosh2012b}. We only note that for embedded convex compact sets
the visualization map is the inverse of the embedding map. Since the directed subdifferential of a convex
function is the embedded convex subdifferential into $\Dn$ (see e.g.~\cite{BaierFarkhiDC,BaiFarRosh2012a,BaiFarRosh2012b}), the visualization of the directed subdifferential 
(which is called \emph{Rubinov subdifferential}) for a convex function coincides with its convex subdifferential.

The directed subdifferential is an object in the space of directed sets defined for every directed subdifferentiable function as follows.

\begin{definition}\label{def:ddd} Let $f:\R^n\to \R$ be directed subdifferentiable, 
then for any fixed $x\in \R^n$ the \emph{directed subdifferential} of $f$ at $x$ can be defined as follows:
\begin{itemize}
  \item[(i)] For $n=1$, define
	\begin{equation}
	\label{eq:n1}
	\dd f(x) := \overrightarrow{[-f'(x;-1);f'(x;1)]} 
                 := (f'(x;l))_{l = \pm 1} \,.
	\end{equation}
  \item[(ii)] For $n\geq 2$, define
  $$
   \dd f(x) := \left(\dd f_l(0), f'(x;l)\right)_{l\in S_{n-1}} \,,
  $$
  where $f_l:\R^{n-1}\to \R$ is defined by $f_l(\cdot) := f'(x;l+\Prj^\trans_{n-1,l}(\cdot))$. 
\end{itemize}
\end{definition}

We extend calculus rules already known for DC and quasidifferentiable functions to directed subdifferentiable functions, and introduce rules for additional operations: pointwise maximum and minimum, multiplication and division.

The next proposition shows that the directed subdifferential of the directional
derivative is the same as the one of the function $f$, thus generalizing~\cite[Proposition~3.15]{BaiFarRosh2010}, \cite[Proposition~5.2]{BaiFarRosh2012b} for DC resp.~QD functions. There is also a close link to the result in~\cite[(35)]{DEM/JEY:1997} in which
the Clarke's subdifferential of the directional derivative yields a smaller subdifferential,
the Michel-Penot subdifferential.

\begin{prop} \label{prop:subd_dir_deriv}
Let $f: \R^n \rightarrow \R$ be directed subdifferentiable in $x \in \R^n$.  Then
    \begin{align*}
       \dd [f'(x;\cdot)](0) & = \dd f(x) \,.
    \end{align*}
\end{prop}
\begin{proof}
Let $\phi(y):= f'(x;y)$. Then by the positive homogeneity of the directional derivative we have
$$
\phi'(0;l) = \lim_{t\downarrow 0 }\frac{f'(x;t l) - f'(x;0)}{t} 
= \lim_{t\downarrow 0 }\frac{t f'(x; l) - 0}{t} = f'(x;l).
$$
Since the directed subdifferential is defined uniquely through the directional derivative, the result is proved for all state dimensions $n$.
\end{proof}


\subsection{Algebraic operations}

Our proofs of the calculus rules rely on the following well-known properties of the directional derivatives
in~\cite[Proposition~3.1 from Sec.~I.3]{DemRub1995}.

\begin{lemma}\label{lem:DirDer} Let $f_1,f_2:\R^n\to \R$ be directionally differentiable at $x\in \R^n$. Then their linear combination, product and quotient (if $f_2(x) \ne 0$) are also directionally differentiable
at this point and the following formulas hold for $\alpha,\beta \in \R$, $l \in \R^n$:
\begin{align*}
(\alpha f_1+\beta f_2)'(x;l) & = \alpha f'_1(x;l)+\beta f'_2(x;l) \,,\\
(f_1\cdot f_2)'(x;l) & = f_1(x)f'_2(x;l)+f_2(x)f'_1(x;l) \,,\\
\left(\frac{f_1}{f_2}\right)'(x;l) & = -\frac{f_1(x)f'_2(x;l)-f_2(x)f'_1(x;l)}{[f_2(x)]^2}
\end{align*}
\end{lemma}

We recall the definitions of algebraic operations on directed sets first. Observe that these rules are designed to be consistent with the corresponding operations on convex sets and pairs of convex compacts.

Linear operations on directed intervals are defined in a natural way and are motivated by the operations for vectors in $\R^2$: for directed intervals $\overrightarrow{[a,b]}$ and $\overrightarrow{[c,d]}$ and any $\a,\b\in \R$
\begin{equation}\label{eq:LinOpDirInt}
\alpha \overrightarrow{[a,b]}+\beta \overrightarrow{[c,d]} := \overrightarrow{[\alpha a+\beta c, \alpha b+\beta d]} \;.
\end{equation}

The \emph{linear operations} are defined recursively on the two components of the
directed sets $\direct{A} = (\direct{A_{n-1}(l)}, a_n(l))_{l \in \Sn}$,
$\direct{B} = (\direct{B_{n-1}(l)}, b_n(l))_{l \in \Sn}$:
      \begin{equation}
       \begin{array}{|r@{\,}c@{\,}l|}
                  \hline \rule{0pt}{3ex}
          \direct{A} + \direct{B} & := & (\direct{A_{n-1}(l)} +
             \direct{B_{n-1}(l)}, a_n(l) + b_n(l))_{l \in \Sn} \;, \\[0.5ex]
          \lambda \cdot \direct{A} & := & (\lambda \cdot
             \direct{A_{n-1}(l)}, \lambda \cdot a_n(l))_{l \in \Sn}
             \hfre (\lambda \in \R) \;, \\[0.5ex]
          \direct{A} - \direct{B} & := & \direct{A} + (-\direct{B})
             = (\direct{A_{n-1}(l)} - \direct{B_{n-1}(l)},
                a_n(l) - b_n(l))_{l \in \Sn} \;. \\[0.5ex]
          \hline
       \end{array}
       \label{Eq_Operat_Dir_sets}
      \end{equation}

We start with the simplest algebraic operation, the sum of two directed subdifferentiable functions which results in the exact sum rule for the directed subdifferential.

\begin{prop}[\bf directed subdifferential of a linear combination]\label{prop:DDSum}
\mbox{} \\
Let $f_1, f_2:\R^n\to \R$ be directed subdifferentiable at $x\in \R^n$ and $\alpha,\beta\in\R$. 

Then $f= \alpha f_1+\beta f_2$ is also directed subdifferentiable and
\begin{equation}\label{eq:SumRule}
\dd f(x) = \alpha\dd f_1(x)+\beta\dd f_2(x) \,.
\end{equation}
\end{prop}
\begin{proof}
From~\cite[Lemma~5.2]{BaiFarRosh2014a} we know that $f$ is directed subdifferentiable. It remains to show \eqref{eq:SumRule}. 

We use the induction argument on the space dimension $n$. We first prove the statement~\eqref{eq:SumRule} for $n=1$.
Using Definition \ref{def:ddd}, Lemma \ref{lem:DirDer}
and properties of operations on directed intervals,  
we have
\begin{align*}
\dd f(x) & = \overrightarrow{[-f'(x;-1),f'(x;1)]}\\
  & = \overrightarrow{[-(\alpha f'_1(x;-1)+\beta f'_2(x;-1)),\alpha f'_1(x;1)+\beta f'_2(x;1)]}\\
  & = \alpha\overrightarrow{[-f'_1(x;-1),f'_1(x;1)]} +\beta \overrightarrow{[-f'_2(x;-1),f'_2(x;1)]} \\  
  & = \alpha \dd f_1(x)+\beta\dd f_2(x) \,.
\end{align*}

Now assume that \eqref{eq:SumRule} holds for $n-1$. We show that it is also valid for $n$. 
By Lemma \ref{lem:DirDer} 
we have for all $l\in \Sn$
\begin{equation*}\label{eq:002}
   f'(x;l) =\alpha f'_1(x;l)+\beta f'_2(x;l) 
   \,, 
\end{equation*}
hence, for all $y\in \R^{n-1}$
$$
f_l(y) = \alpha (f_1)_l(y)+\beta (f_2)_l(y)
$$
and by the directed subdifferentiability of $(f_i)_l$, $i=1,2$,
and~\cite[Lemma~5.2]{BaiFarRosh2014a}, we conclude
that $f_l$ is directed subdifferentiable at $y=0$. 
Therefore, by the induction assumption for all $l\in \Sn$ we get
\begin{equation}\label{eq:003}
\dd f_l(0) = \alpha \dd (f_1)_l(0) + \beta \dd (f_2)_l(0)
\end{equation}
Now, from Definition \ref{def:ddd}, 
\eqref{eq:003} and the definition of addition and scalar multiplication on directed sets 
we have
\begin{align*}
\dd f(x) & 
= \left(\alpha\dd (f_1)_l + \beta\dd (f_2)_l, \alpha f'_1(x;l)+\beta f'_2(x;l)\right)_{l\in  \Snvar{n}}\\
 & = \alpha\left(\dd (f_1)_l, f'_1(x;l)\right)_{l\in  \Snvar{n}}+\beta \left(\dd (f_2)_l,f'_2(x;l)\right)_{l\in  \Snvar{n}}\\ 
& =\alpha \dd f_1(x)+\beta \dd f_2(x) \,.
\end{align*}
\end{proof}

Observe that the multiplication with negative scalars is possible in the last
proposition. 




Next we consider the product of directed subdifferentiable functions and derive the Leibniz rule for
the directed subdifferential.

\begin{prop}[\bf directed subdifferential of product and ratio]\label{prop:DDProd}
\mbox{} \\
Let $f_1, f_2:\R^n\to \R$ be directed subdifferentiable at $x\in \R^n$. 

Then $f= f_1\cdot f_2$ is also directed subdifferentiable at $x$ and
$$
\dd f(x) = f_1(x) \dd f_2(x) +f_2(x) \dd f_1(x).
$$
Moreover, if $f_2(x)\neq 0$, then $f= f_1/ f_2$ is also directed subdifferentiable and
$$
\dd f(x) = -\frac{1}{[f_2(x)]^2}\left[f_1(x)\dd f_2(x)- f_2(x) \dd f_1(x)\right] \,.
$$ 
\end{prop}
\begin{proof}
We will use the calculus rule in~Lemma \ref{lem:DirDer} for the directional derivative of a product
of functions. The proof proceeds by induction with respect to $n$.
For the directed intervals we have
\begin{align*}
\dd f(x) & = \overrightarrow{[-f'(x;-1),f'(x;1)]} \\
    & = \overrightarrow{[-f_1(x)f'_2(x;-1)-f_2(x)f'_1(x;-1),f_1(x)f'_2(x;1)+f_2(x)f'_1(x;1)]}\\
    & = f_1(x) \overrightarrow{[-f'_2(x;-1),f'_2(x;1)]}+f_2(x)\overrightarrow{[-f'_1(x;-1),f'_1(x;1)]}\\
    & = f_1(x) \dd f_2(x)+f_2(x)\dd f_1(x) \,.
\end{align*}
Assume now that the proposition is true for $n=k-1$ for some $k\geq 2$. For $n=k$, by 
Lemma~\ref{lem:DirDer}
we have for all $l\in \R^n$
\[
f'(x;l) = f_1(x) f'_2(x;l)+f_2(x) f'_1(x;l) \,,
\]
and hence for all $y\in \R^{n-1}$
$$
f_l(y) = f_1(x)(f_2)_l(y)+f_2(x)(f_1)_l(y) \,.
$$
By Proposition~\ref{prop:DDSum} 
this yields
\begin{equation}\label{eq:005}
\dd f_l(0) =  f_1(x)\dd (f_2)_l(0)+f_2(x)\dd (f_1)_l(0) \,.
\end{equation}
Therefore, we have
\begin{align*}
   & \dd f(x) = \left(\dd f_l(0), f'(x;l)\right)_{l\in S_{n-1}}\\
   = \mbox{} & \left(f_1(x)\dd (f_2)_l(0)+f_2(x)\dd (f_1)_l(0),f_1(x) f'_2(x;l)+f_2(x) f'_1(x;l)\right)_{l\in S_{n-1}}\\
   = \mbox{} & f_1(x)\dd f_2(x)+f_2(x)\dd f_1(x) \,.
\end{align*}

The proof of the expression for the directed subdifferential of a ratio is analogous and is omitted.
\end{proof}

\subsection{Pointwise maximum and minimum}

To describe the calculus rules for pointwise maximum and minimum, we first ought 
to define the corresponding operations on directed sets. The supremum and infimum operations 
have already been defined for pairs of directed sets in~\cite{BaiFar2001}
and are based on the order in this space, see~\cite[Definitions~3.5 and 4.6]{BaiFar2001},~\cite{BaiFarRosh2012b}. 
Note that the notion of the supremum of two directed sets is consistent with 
the convex hull of the union of two convex sets.

\begin{definition}[\bf supremum and infimum of directed sets]
 \mbox{}
 \label{def:sup_inf} 
\begin{itemize}
 \item[(i)] Let $n=1$, $I =\{1,\dots, p\}$ and 
$\overrightarrow A_i = \overrightarrow{[\alpha_i^-,\alpha_i^+]}$, $i\in I$, 
be directed intervals. Define the \emph{supremum} and \emph{infimum}
     $$
        \sup_{i\in I}\{\overrightarrow A_i\} = \overrightarrow{[\min_{i\in I} \alpha_i^-, \max_{i\in I} \alpha_i^+]} \,;
\qquad
        \inf_{i\in I}\{\overrightarrow A_i\} = \overrightarrow{[\max_{i\in I} \alpha_i^-, \min_{i\in I} \alpha_i^+]} \,.
     $$
   \item[(ii)] For $n\geq 2$, $I =\{1,\dots, p\}$ and  directed sets $\{\overrightarrow A^i_n\}_{i\in I}\subset \Dn$ with
   $$
   \overrightarrow A_n^i = \left(\overrightarrow{A^i_{n-1}(l)}, a_n^i(l)\right)_{l\in \Sn} \,,
   $$
   we define the \emph{supremum} and \emph{infimum}
   $$
   \sup_{i\in I}\{\overrightarrow A_n^i\} = \left(\sup_{i\in I(l)}\{\overrightarrow{A^i_{n-1}(l)}\} ,\max_{i\in I}\{a_n^i(l)\}\right)_{l\in \Sn} \,,
   $$
   $$
   \inf_{i\in I}\{\overrightarrow A_n^i\} = \left(\inf_{i\in J(l)}\{\overrightarrow{A^i_{n-1}(l)}\} ,\min_{i\in I}\{a_n^i(l)\}\right)_{l\in \Sn} \,,
   $$
   where $I(l) = \{i\,|\, a_n^i(l) = \max_{k\in I}  a_n^{k}(l) \}$, $J(l) = \{i\,|\, a_n^i(l) = \min_{k\in I}  a_n^{k}(l) \}$.
\end{itemize}
\end{definition}

We will use a property of the min/max operations on directional derivatives, proved in~\cite[Corollary~3.2 in Sect.~I.3]{DemRub1995}:

\begin{lemma}\label{lem:DirDerMaxMin} Let $f_i:\R^n\to \R$, $i\in I = \{1,2,\dots, p\},$ 
be directionally differentiable at $x\in \R^n$ for all $l\in \Sphn$. 

Then pointwise maximum function ${f_{\max}(\cdot) := \max_{i\in I}f_i(\cdot)}$ 
and the pointwise minimum function $f_{\min}(\cdot) := \min_{i\in I} f_i(\cdot)$ are also directionally differentiable at $x$ with
$$
f_{\max}'(x;l) = \max_{i\in I(x)} f_i'(x;l);\quad f_{\min}'(x;l) = \min_{i\in J(x)} f_i'(x;l) \,,
$$
where $I(x) := \{i\in I\,|\, f_i(x) = f_{\max}(x)\}$, $J(x) := \{i\in I\,|\, f_i(x) = f_{\min}(x)\}$.
\end{lemma}


\begin{prop}[\bf directed subdifferential of a pointwise maximum] \label{prop:SubdMax}
Let $f_{i}:\R^n\to \R$ be directed subdifferentiable at $x\in \R^n$
for $i \in I=\{1,\ldots,p\}$.  

Then the pointwise maximum function $f:\R^n\to \R$, $f(x) = \max_{i \in I} f_i(x)$ 
is also directed subdifferentiable at $x$ and
\begin{equation}\label{eq:pSubMax01}
\dd f(x) = \sup_{i\in I(x)}\{\dd f_i(x)\} \,,
\end{equation}
where $I(x) = \{ i \in I\,|\, f(x) = f_i(x) \}$ is the set of active indices.
\end{prop}
\begin{proof}
Observe that $f$ is directed subdifferentiable by~\cite[Lemma~5.3]{BaiFarRosh2014a}, 
hence, it remains to show \eqref{eq:pSubMax01}. 
As usual, we use the induction argument. \\
For $n=1$: 
Let $f_1,\dots, f_p:\R^1\to \R^1$ be directed subdifferentiable 
at $x\in \R$, and denote $\dd f_i(x) = \overrightarrow{[d_i^-,d_i^+]}$. Then by 
Lemma \ref{lem:DirDerMaxMin} $f$ is directionally differentiable with a finite directional derivative
$$
f'(x;l) = \max_{i\in I(x)}\{f'(x;l)\} \quad \text{for } l\in \{-1,1\} \,,
$$
where $I(x) = \{ i \in I \,|\, f(x) = f_i(x)\}$. By Definitions \ref{def:ddd} and \ref{def:sup_inf}, we can write
\begin{align*}
\dd f(x) & = \overrightarrow{[-f'(x;-1),f'(x;1)]}\\
& =  \overrightarrow{\left[-\max_{i\in I(x)}\{f'(x;-1)\},\max_{i\in I(x)}\{f'(x;1)\}\right]}
  =  \sup_{i\in I(x)}\{\dd f_i(x)\} \,.
\end{align*}
We thus have our induction base for $n=1$.

Now assume that the result is true for $n - 1$, we show that it is also true for $n$.
Notice that for given $y \in \R^{n-1}$
$$
f_l(y) = f'(x;l+\Prj^\trans_{n-1,l}y)= \max_{i\in I(x)}f_i'(x;l+\Prj^\trans_{n-1,l}y)= \max_{i\in I(x)}(f_i)_l(y) \,.
$$
Since for any $i=1,...,p$, the function $(f_i)_l$ is directed subdifferentiable by Definition \ref{def:dsf}, and the relevant max-function is directed subdifferentiable by~\cite[Lemma~5.3]{BaiFarRosh2014a} and our induction assumption $\dd f_l(0) = \sup_{i\in I(x)} \dd
(f_i)_l(0)$. Therefore, 
$$
\dd f(x) = (\dd f_l(0), f'(x;l))_{l\in \Snvar{n}} = (\sup_{i\in I(x)} \dd
(f_i)_l(0), \max_{i\in I(x)}f_i'(x;l))_{l\in \Snvar{n}} \,.
$$
Since $\dd f_i(x) = (\dd (f_i)_l(0), f_i'(x;l))_{l\in \Snvar{n}}$, this yields
$$
\dd f (x) = \sup_{i\in I(x)}\dd f_i(x) \,.
$$
\end{proof}
Note that in general only inclusion results for other subdifferentials are available  for 
pointwise maximum or minimum.
See~\cite{Rosh2010} for a discussion of this issue and an improved inclusion rule
for min-functions and the Mordukhovich subdifferential.

The following statement is proved similarly to the previous one.
\begin{prop}[\bf directed subdifferential of a pointwise minimum]\label{prop:SubdMin}
Let $f_{i}:\R^n\to \R$ be directed subdifferentiable at $x\in \R^n$
for $i \in I=\{1,\ldots,p\}$.  

Then the pointwise minimum function $f:\R^n\to \R$, $f(x) = \min_{i\in I} f_i(x)$ is also directed subdifferentiable at $x$. Moreover,
$$
\dd f(x) = \inf_{i\in I(x)}\{\dd f_i(x)\} \,,
$$
where $I(x) = \{ i \in I \,|\, f(x) = f_i(x) \}$ is the set of active indices.
\end{prop}

\section{Optimality Conditions}
\label{secOptCond}

Recall that the classical nonsmooth optimality conditions for a nonsmooth functions can be stated in terms of the directional derivative: if the function attains a local minimum at $x$, then $f'(x;l)\geq 0$ for all directions $l \in \Sn$.
For a convex function this condition translates into the inclusion $0\in \partial f(x)$, 
where $\partial f(x)$ is the Moreau-Rockafellar subdifferential of $f$. 
For differentiable convex  functions this subdifferential is a singleton and the 
latter condition is just the Fermat rule $f'(x)=\nabla f(x)^\trans=0$.
Here we show  that the optimality conditions for directed subdifferentiable
functions generalize both the Fermat rule for differentiable functions and the 
above inclusion for convex functions. In the space of directed sets, we replace the
relations of set inclusion by inequalities expressing the partial order in this space.
Similarly, the necessary conditions for QD functions in~\cite[(28)]{DEM/JEY:1997} 
formulated with the Michel-Penot subdifferential have a close link with the 
visualization of the directed subdifferential.

\begin{remark} \label{rem:diff_fct}

   Let $f: \R^n \rightarrow \R$ be Fr\'{e}chet differentiable. 
   Then, $f$ is directed subdifferentiable with
    \begin{align*}
       \dd f(x) & = \Embed{n}{ \{ \nabla f(x)^\trans \} }  = \Embed{n}{ \{  f'(x) \} }\,.
    \end{align*}
   %
%
%
   By the assumption and \cite[Subsec.~III.2.1, 1.]{DemRub1995}, $f$ is quasidifferentiable with
    \begin{align*}
       f'(x;l) & = \nabla f(x)^\trans l = \supf{l}{ \{\nabla f(x)^\trans\} } - \supf{l}{ \oldminus \{0\} } \,, 
    \end{align*}
   and the quasidifferential is given by the ordered pair of two singleton convex sets,
       $\QuSubD{f}{x} 
			           = [ \{\nabla f(x)^\trans\},\{0\} ] \,$.
%
   Since quasidifferentiable functions are a special case of directed subdifferentiable ones, we can apply the results in
   \cite[Subsec.~4.2]{BaiFarRosh2014a} so that
    \begin{align*}
       \dd f(x) & = \dd_{QD} f(x) = \Embed{n}{ \{\nabla f(x)^\trans\} } \,.
    \end{align*}
   %
\end{remark}

For a directed subdifferentiable function, we can state optimality conditions in terms of the order relations between the embedded zero $\direct 0$ and the directed subdifferential, which can formally be written as $\direct{0}\leq \dd f(x)$. To make this expression precise we first recall the definitions of the order relation on directed sets, and of the directed zero.

The partial order on the space of directed sets can be introduced using the notion of supremum discussed in the previous section:
$$
\direct A \leq \direct B \quad \Leftrightarrow \quad \direct B = \sup\{ \direct A,\direct B \},
$$
which can be made explicit as follows (see \cite[Definition~4.6]{BaiFar2001}).

\begin{definition} For two $n$-dimensional directed sets $\direct A$ and $\direct{B}$ we 
define the \emph{partial order} and write $\direct{A}\leq \direct{B}$ if  
\label{def:order}
\begin{itemize}
\item[(i)] for every $l\in \Sn$ one has $a_n(l) \leq b_n(l)$;
\item[(ii)] whenever $n\geq 2$, and $a_n(l) = b_n(l)$, one has $\direct{A_{n-1}(l)}\leq \direct{B_{n-1}(l)}$.
\end{itemize}
\end{definition}

We can define a directed zero as a trivial directed set with all components being recursively zero. This set corresponds to the embedding of the singleton $\{0_n\}$ in the space of directed sets.

\begin{definition}[\bf directed zero] For $n=1$, the \emph{zero directed interval} \emph{is} 
$$
\direct 0_1 = \overrightarrow{[0,0]};
$$ 
when $n>1$, we define the \emph{directed zero} as 
$$
\direct 0_n = (\direct 0_{n-1}, 0)_{l\in \Sn}.
$$
\end{definition}

Whenever the dimension is clear from the context, we denote the directed zero by $\direct 0$.
Clearly, this is the neutral element for the addition in $\Dn$.

The following technical result is helpful in proving the optimality conditions in terms of directed sets.

\begin{lemma}\label{lem:tech001} Let $f:\R^n \to \R$ be a directed subdifferentiable function. \\
If $f$ attains a local minimum at $x$, then whenever $f'(x;l)=0$ for some $l\in \Sn$, the function  $f_l(\cdot) = f'(x;l + \Prj^\trans_{n-1,l}(\cdot))$ attains a global minimum at $0$.
\end{lemma}
\begin{proof} Assume the contrary, i.e.~$f'(x;l)=0$ and that there exists $y\in \R^{n-1}$ such that $f_l(y)<f_l(0)$. Hence,
$$
f'(x;l+\Prj^\trans_{n-1,l}(y))<f'(x;l) = 0
$$
follows which is impossible, since $x$ is a local minimum, and the directional derivative $f'(x;l)$ is must be non-negative for all values of $l$.
\end{proof}

\begin{prop}[\bf necessary optimality condition]
\label{prop:nec_opt_cond}%
\label{prop:ineq} Let $f:\R^n\to \R$ be directed subdifferentiable. If $\widehat{x}$ is a local minumum of $f$, then
\begin{equation}\label{eq:DirectedMinimumNecessary}
     \direct{0}  \leq \dd f(\widehat{x}) \,,
\end{equation}
  whereas if $\widehat x$ is a local maximum, then
  \[
     \direct{0}  \leq -\dd f(\widehat{x}).
  \]
\end{prop}
\begin{proof} We prove the necessary condition for the local minimum only, as the necessary condition for the local maximum is symmetric: it can be proved by considering the local minima of $-f$, and applying the calculus rules to the directed subdifferential.

By the definition of partial order, \eqref{eq:DirectedMinimumNecessary} is equivalent to showing that if $x$ is a local minimum of $f$, then 
\begin{equation}\label{eq:DirectedMinimumNecessary01}
\dd f(x)  = \sup\{\direct{0},\dd f(x)\}.
\end{equation}

First of all, observe that the statement is true for $n=1$: if a function $f:\R\to \R$ attains its local minimum at $x$, then $f'(x;-1)> 0$, $f'(x;1)>0$, and hence
\begin{align*}
\dd f(x) 
& = \direct{[-f'(x;-1), f'(x;1)]} \\
& = \direct{[\min\{0, -f'(x;-1)\}, \max\{0,f'(x;1)\}]}
 = \sup\{\direct{0},\dd f(x)\}.
\end{align*}

To prove the statement for $n\geq 2$, we proceed by induction. Assume that the statement is true for $n-1$. Let $f$ have a local minimum at $x$, then $f'(x;l)\geq 0$ for all $l$, and hence 
\begin{equation}\label{eq:tech002}
\max\{0, f'(x;l)\} = f'(x;l)\qquad \forall \, l\in \Sn. 
\end{equation}
Whenever $f'(x;l)=0$, $f_l$ attains the minimum at $0_{n-1}$ by Lemma~\ref{lem:tech001}. Hence, by the induction assumption we have 
\begin{equation}\label{eq:tech003}
\sup \{ \direct{0_{n-1}}, \dd f_l(0) \} = \dd f_l(0)\qquad \forall \, l: \, f'(x;l)=0.
\end{equation}
We now have \eqref{eq:DirectedMinimumNecessary01} from \eqref{eq:tech002}--\eqref{eq:tech003} and the definition of supremum.
\end{proof}
In the particular case when $f$ is differentiable, its directed subdifferential
is an embedded singleton containing the gradient and it is easy to see that the inequality
\eqref{eq:DirectedMinimumNecessary} between two embedded singletons
is reduced to equality, which is equivalent to the Fermat rule.

\section{Chain Rule and Mean-Value Theorem}
\label{SecChainMeanVal}

The developments of this section are based on the following result on the
composition of a nonsmooth and a smooth function which is stated in
\cite[Sec.~I.3, Theorem~3.3]{DemRub1995} for a more general outer function 
which is Hadamard differentiable. Here, we only need a subclass of Hadamard
differentiable functions (cf.~\cite[Sec.~I.3, Proposition~3.2]{DemRub1995}).
 
\begin{lemma}\label{lem:DirDerCompos} 
  Let $h:\R^n\to \R^m$ be directionally differentiable at $x \in \R^n$ 
  and $g: \R^m \to \R$ be directionally differentiable and Lipschitz continuous 
  at $h(x)$. Then the composition $f = g \circ h$ is directionally differentiable at $x$
  with
   \begin{align*}
      f'(x;l) & = g'(h(x); h'(x; l)) \quad (l \in \Sn) \,.
   \end{align*}
\end{lemma}

We prove here a chain rule for the directed subdifferential of the composition of two functions, 
with the first function (hence, the composition) defined on an interval. We 
apply this rule to derive a mean-value theorem for the directed subdifferential.

There is a difficulty to extend the chain rule to compositions of multivariate functions,
since the expected product of a matrix (the Jacobian) and a directed set (the directed subdifferential)
has been defined yet only for directed sets which are limits of differences of embedded
convex sets. 

The next proposition allows to change the inner function by its first Taylor approximation
within a composition without changing the directed subdifferential. 
This result can be seen as a motivation to demand Ioffe's axiom~($\mbox{SD}_7$)
for subdifferentials in~\cite[Chap.~2, Subsec.~1.5]{IOF:2000a}.

\begin{prop} \label{prop:dir_subdiff_taylor_approx}
   Let $x^0 \in \R^n$ and $f: \R^n \rightarrow \R$ be in the form
    \begin{align*}
       f(x) & = (g \circ \varphi)(x) \quad (x \in \R^n) \,,
    \end{align*}
   where $\varphi: \R^n \to \R^m$ is 
   Fr\'{e}chet differentiable in $x^0$,
   $g: \R^m \to \R$ is directed subdifferentiable and locally Lipschitz in 
   $\varphi(x^0)$. 
   Let $\widetilde{\varphi}(y) = \varphi(x^0) + \varphi'(x^0) (y-x^0)$
   for $y \in \R^n$. 

   Then, the directed subdifferential of $f$ and 
   $\widetilde{f} = g \circ \widetilde{\varphi}$ 
   coincide at $x^0$, i.e.
    \begin{align*}
       \dd f(x^0) & = \dd \widetilde{f}(x^0) \,.
    \end{align*}
\end{prop}
\begin{proof}
   Clearly, $\widetilde{\varphi}$ is differentiable with 
    \begin{align*}
       \widetilde{\varphi}(x^0) & = \varphi(x^0) \,, \quad
         \widetilde{\varphi}'(x^0) = \varphi'(x^0)
    \end{align*}
   as the first-order Taylor approximation of $\varphi$.

   We claim that the directed subdifferential of $g \circ \widetilde{\varphi}$ 
   in $x^0$ coincides with the one of
   $g \circ \varphi$. Indeed, by the differentiability of $\varphi$ 
   in $x^0$ and the Lipschitz continuity of $g$,
   we can apply Lemma~\ref{lem:DirDerCompos} which shows
    \begin{align*}
       \widetilde{f}'(x^0; l) 
         & = g'(\widetilde{\varphi}(x^0); \widetilde{\varphi}'(x^0) l) 
           = g'(\varphi(x^0); \varphi'(x^0) l) = f'(x^0; l) \quad (l \in \R^n)
    \end{align*}
   and therefore the equality of
   the (second) component of $\dd \widetilde{f}(x^0)$ 
   and $\dd f(x^0)$. \\
   In its first component for $n > 1$, the directed subdifferential of the function $\widetilde{f}_l$ appears.
   For this function we can use the previous equality which yields 
    \begin{align*}
       \widetilde{f}_l(y) & = \widetilde{f}'(x^0; l + \Prj^\trans_{n-1,l} y)
         = f'(x^0; l + \Prj^\trans_{n-1,l} y) = f_l(y)
    \end{align*}
   for $y \in \R^{n-1}$. Hence, by induction
    \begin{align*}
       \dd \widetilde{f}(x^0) & = \dd f(x^0) \,.
    \end{align*}
\end{proof}

The following proposition is a special form of a chain rule for
a composition of a directed subdifferentiable and a continuously differentiable function.
It is used later 
in the proof of the mean-value theorem. 

\begin{prop} \label{prop:spec_chain_rule}
   Let $t_0 \in \R$ and $f: \R \rightarrow \R$ be in the form
    \begin{align*}
       f(t) & = (g \circ \varphi)(t) \quad (t \in \R)
    \end{align*}
   where $\varphi: \R \to \R^m$ is differentiable at $t_0$ 
   and $g: \R^m \to \R$ is directed subdifferentiable and locally Lipschitz at 
   $\varphi(t_0)$. 

Then, $f: \R \rightarrow \R$ is directed subdifferentiable at $t_0$ with
\begin{align}\label{eq:ChainRuleOneDim}
\dd f(t_0) = \overrightarrow{[-g'(\varphi(t_0); -\varphi'(t_0)), g'(\varphi(t_0); \varphi'(t_0))]}.
\end{align}
\end{prop}

\begin{proof}
The proof follows directly from Lemma~\ref{lem:DirDerCompos} and the definition of the directed subdifferential.
\end{proof}

\begin{remark}
Observe that the expression \eqref{eq:ChainRuleOneDim} cannot be simply written as 
$\dd f(t_0) = \varphi'(t_0) \dd g(\varphi(t_0))$,  
since then $\varphi'(t_0)$ has to play the role of a (yet undefined) linear operator 
from $\mathcal{D}(\R^m)$ to $\mathcal{D}(\R)$ satisfying a property analogous to Ioffe's 
axiom~($\mbox{SD}_7$) in \cite{IOF:2000a}. Of course, this linear operator may be easily defined in the
case when $g$ is differentiable and the directed sets involved are embedded singletons.
%
%
\end{remark}

In the special case when the inner function $\varphi$ is affine, we get the following 
formula for the directed subdifferential of the composition from the previous
proposition.

\begin{corollary} \label{cor:affine_chain_rule}
   Let $x^0, x^1 \in \R^n$ and let $g: \R^n \rightarrow \R$ be directed subdifferentiable
   and Lipschitz continuous on the line segment $\co\{x^0, x^1\} = \{ x^0 + t (x^1 - x^0) \,|\, t \in [0,1] \, \}$.
   Then, $f: [0,1] \rightarrow \R$ with $f(t) = g(x^0 + t(x^1 - x^0))$ is directed subdifferentiable 
   for every $t \in [0,1]$ with
    \begin{align*}
       \dd f(t) & = \direct{ [-g'(x^0 + t(x^1 - x^0); -(x^1 - x^0)), g'(x^0 + t(x^1 - x^0); x^1 - x^0)] }\,.
    \end{align*}
\end{corollary}
%

The following mean-value theorem  is an analogue
of~\cite[Proposition~1.115]{Morduk2006Book1} for the directed subdifferential.
In~\cite{Morduk2006Book1} the mean-value theorem is formulated with the
basic symmetric subdifferential which has a very close connection to the
visualization of the directed subdifferential in $\R^2$, see~\cite{BaiFarRosh2010}.

\begin{theorem}[\bf mean-value theorem] \label{theo:mean_value} 
   \mbox{} \\
   Let $x^0,x^1 \in \R^n$ and $g: \R^n \rightarrow \R$ be directed subdifferentiable 
   on the open segment $A=
   \{ x^0 + t (x^1 - x^0) \,|\, t \in (0,1) \, \}$
   and Lipschitz continuous on $\cl A$. 

   Then there exists $\hat{t} \in (0,1)$ such that
    \begin{equation} \label{eq:mvt_1}
       \Embed{1}{ \{ g(x^1) - g(x^0) \} } 
         \leq \dd g(x^0 + \cdot\, (x^1 - x^0))(\hat{t}\,) \,.
    \end{equation}
   %
\end{theorem}
\begin{proof}
   Define the affine function $\varphi: [0,1] \to \R^n$ 
   by $\varphi(t) = x^0 + t (x^1 - x^0)$,
   $f_1(t)=g(\varphi(t))$
   and $f_2(t)=t (g(x^0) - g(x^1))$ for $t \in [0,1]$. 
   Then, $\varphi([0,1]) = \cl A$ and we set the function
   $f: [0,1] \to \R$ by $f(t)=f_1(t)+f_2(t)$.

   Clearly, $\varphi,f_2$ are $C^1$ on $[0,1]$ and 
   $f$ is directed subdifferentiable by Proposition~\ref{prop:DDSum} 
   and Corollary~\ref{cor:affine_chain_rule} with
    \[
       \dd f(t)  = \dd f_1(t) + \dd f_2(t) \,.
    \]
   %
   %

   The function $f$ is also continuous on $[0,1]$ 
   and satisfies $f(0) = f(1) = g(x^0)$. 
   Therefore, either $f$ is constant on $[0,1]$ or
   it attains its global extremum (minimum or maximum)  at some point in $(0,1)$ 
   by the classical Weierstra\ss{} theorem.
   In both cases, $f$ has a local extremum in the open interval $(0,1)$ and by 
   Proposition~\ref{prop:nec_opt_cond} there exists a point $\hat{t}$ such that 
    \begin{equation}
       \label{eq:o_in_vis}
       \direct{0} \leq \dd f(\hat{t}\,) \,.
    \end{equation}
   For the function $f_1$ we apply Proposition~\ref{prop:spec_chain_rule} to obtain 
   %
    \begin{align}
       \dd f_1(\hat{t}\,) 
       & = \overrightarrow{ [-g'(\varphi(\hat{t}\,); -\varphi'(\hat{t}\,)), 
                              g'(\varphi(\hat{t}\,); \varphi'(\hat{t}\,))] } 
           \nonumber \\
       & = \overrightarrow{ [-g'(\widehat{x}; -(x^1 - x^0)), 
                              g'(\widehat{x}; x^1 - x^0)] } \,,
           \label{eqar:f1}
    \end{align}
   where $\widehat{x}=x^0+ \hat{t}(x^1-x^0)$. 
   Since the function $f_2(t) = t (g(x^0) - g(x^1))$ is differentiable, 
   we apply Remark~\ref{rem:diff_fct}:
    \begin{align*}
       \dd f_2(t) & = \Embed{1}{ \{ f_2'(t) \} } = \Embed{1}{ \{ g(x^0) - g(x^1) \} }
    \end{align*}
   Hence,
    \begin{align*}	
       \dd f(\hat{t}\,) & = \dd f_1(\hat{t}\,) + \dd f_2(\hat{t}\,) \\
       & = \Embed{1}{ \{ g(x^0) - g(x^1) \} } 
         +\overrightarrow{[-g'(\hat{x}; -(x^1 - x^0)), g'(\hat{x}; x^1 - x^0)] }. 
    \end{align*}
   The inverse of an embedded scalar in $\Done$ is the embedded negative scalar
   so that~\eqref{eq:o_in_vis} transfers to the asserted inequality
    \begin{align*}	
        \Embed{1}{ \{ g(x^1) - g(x^0) \} }
          & \leq \overrightarrow{[-g'(\hat{x}; -(x^1 - x^0)), g'(\hat{x}; x^1 - x^0)] } \,.
    \end{align*}
\end{proof}

\section{Conclusions}


While in \cite{BaiFarRosh2014a} we introduced the class of directed subdifferentiable 
functions, this paper is devoted to calculus rules for their directed subdifferential.
Let us stress that the calculus rules are given by equalities, even for the sum rule 
and for  maximum or minimum of directed subdifferentiable functions,
without any delta-convex structure of the functions or their directional
derivaties.

Thus, we obtain exact calculus rules, similar formulas to those for the quasidifferential of
Demyanov/Rubinov, and avoid inclusions in the calculation of the directed subdifferential. 
Also the chain rule is extended here from the class of quasidifferentiable to 
directed subdifferentiable functions. The visualization of the directed subdifferential 
(the so-called Rubinov subdifferential) is related not only to convex subdifferentials, 
as the Dini and the Michel-Penot subdifferential, but also to the (nonconvex) basic
subdifferential of Mordukhovich \cite{BaierFarkhiDC,BaiFarRosh2010}. 
The mean-value 
theorem and the numerical calculation of the directed subdifferential may open the
way to numerical algorithms for nonsmooth optimization problems.
More concrete applications of the directed (or Rubinov) subdifferential to nonsmooth 
optimization problems and algorithms are the subject of a forthcoming work.


\section*{Acknowledgements}
   We thank Wolfgang Achtziger for motivating us to study the mean-value theorem. 
   This work is partially supported by The Hermann Minkowski Center for Geometry 
   at Tel Aviv University, Tel Aviv, Israel. 



\def\cprime{$'$} \def\cydot{\leavevmode\raise.4ex\hbox{.}}
  \def\cfac#1{\ifmmode\setbox7\hbox{$\accent"5E#1$}\else
  \setbox7\hbox{\accent"5E#1}\penalty 10000\relax\fi\raise 1\ht7
  \hbox{\lower1.15ex\hbox to 1\wd7{\hss\accent"13\hss}}\penalty 10000
  \hskip-1\wd7\penalty 10000\box7} \def\Dbar{\leavevmode\lower.6ex\hbox to
  0pt{\hskip-.23ex \accent"16\hss}D} \def\dbar{\leavevmode\hbox to
  0pt{\hskip.2ex \accent"16\hss}d}

\end{document}